\documentclass[12pt,reqno]{amsart}

\usepackage{amsmath,amsfonts,amsbsy,amsgen,amscd,mathrsfs,amssymb,amsthm}
\usepackage{enumerate,mathtools}
\usepackage{cancel,soul}
\usepackage[usenames,dvipsnames]{xcolor}
\usepackage[colorlinks=true,citecolor=blue,linkcolor=blue]{hyperref}
\usepackage{braket}
\usepackage{tikz}
\usepackage{comment}
\usetikzlibrary{shadings}

\newtheorem{thm}{Theorem}[section]
\newtheorem{lem}[thm]{Lemma}

\newtheorem{prop}[thm]{Proposition}
\newtheorem{cor}[thm]{Corollary}
\newtheorem{rem}[thm]{Remark}

\theoremstyle{definition}

\newtheorem{defn}[thm]{Definition}

\numberwithin{equation}{section} 
\numberwithin{figure}{section}
\numberwithin{table}{section}

\newcommand{\bE}{\mathbf{E}}
\newcommand{\eps}{\varepsilon}

\newcommand{\bP}{\mathbf{P}}

\newcommand{\Om}{\Omega}

\newcommand{\tr}{\textnormal{tr}}

\newcommand{\com}{\mathbb{C}}

\newcommand{\pol}{A}

\newcommand{\fc}{\widehat{A}}

\newcommand{\bs}{\textbf{s}}

\newcommand{\real}{\mathbb{R}}

\newcommand{\im}{\textnormal{Im}}

\newcommand{\Br}{\textnormal{Br}}

\newcommand{\qBr}{\textnormal{qBr}}

\newcommand{\cF}{\mathcal{F}}

\newcommand{\all}{\textnormal{all}}

\begin{document}

\title[Noncommutative  Bohnenblust--Hille inequalities ]{Noncommutative Bohnenblust--Hille inequalities}

\author{Alexander Volberg}
\address{(A.V.) Department of Mathematics, MSU, 
East Lansing, MI 48823, USA, and Hausdorff Center of Mathematics}
\email{volberg@math.msu.edu}

\author{Haonan Zhang}
\address{(H.Z.) Department of Mathematics, University of California, Irvine, CA 92617, USA and Institute of Science and Technology Austria (ISTA), Am Campus 1, 3400 Klosterneuburg, Austria
}
\email{haonanzhangmath@gmail.com}

\begin{abstract}
Bohnenblust--Hille inequalities for Boolean cubes have been proven with dimension-free constants that grow subexponentially in the degree \cite{defant2019fourier}. Such inequalities have found great applications in learning low-degree Boolean functions \cite{eskenazis2022learning}. Motivated by learning quantum observables, a qubit analogue of Bohnenblust--Hille inequality for Boolean cubes was recently conjectured in \cite{RWZ22}. The conjecture was resolved in \cite{CHP}. In this paper, 
we give a new proof of these Bohnenblust--Hille inequalities for qubit system with constants that are dimension-free and of exponential growth in the degree. As a consequence, we obtain a junta theorem for low-degree polynomials. Using similar ideas, we also study learning problems of low degree quantum observables and Bohr's radius phenomenon on quantum Boolean cubes. 
\end{abstract}

\subjclass[2010]{%46B10, 46B09; 46B07; 60E15
47A30, 81P45, 06E30}

\keywords{Bohnenblust--Hille inequality, Boolean cubes, quantum observables, Pauli matrices, quantum learning, juntas, Bohr radius}

\thanks{ The research of A.V. is supported by NSF DMS-1900286, DMS-2154402 and by Hausdorff Center for Mathematics. H.Z. is supported by the Lise Meitner fellowship, Austrian Science Fund (FWF) M3337. This work is partially supported by NSF  DMS-1929284 while both authors were in residence at the Institute for Computational and Experimental Research in Mathematics in Providence, RI, during the Harmonic Analysis and Convexity program.}

\maketitle

\section{Introduction}
In 1930, Littlewood \cite{littlewood30} proved that for any $n\ge 1$ and any bilinear form $B:\com^n\otimes \com^n\to \com$, we have
\begin{equation}\label{ineq:littlewood}
	\left(\sum_{i,j}|B(e_i,e_j)|^{4/3}\right)^{3/4}\le \sqrt{2}\|B\|,
\end{equation}
where $\{e_j,1\le j\le n\}$ is the canonical basis of $\com^n$ and $\|B\|$ denotes the norm of the bilinear form, i.e. 
\begin{equation*}
	\|B\|:=\sup\{|B(x,y)|:x,y\in \com^n, \|x\|_\infty\le 1, \|y\|_\infty\le 1\}.
\end{equation*}
Here $4/3$ is optimal and \eqref{ineq:littlewood} is known as {\em  Littlewood's 4/3 inequality}. 
Right after Littlewood's proof of \eqref{ineq:littlewood}, Bohnenblust and Hille \cite{bohnenblust1931absolute} extended this result to multilinear forms: For any $d\ge 1$, there exists a constant $C_d>0$ depending only on $d$ such that for any $n\ge 1$ and any $d$-linear form $B:\com^n\times \cdots \times \com^n\to \com$ we have
\begin{equation}\label{ineq:bh 1931}
	\left(\sum_{i_1,\dots, i_d=1}^{n}|B(e_{i_1},\dots, e_{i_d})|^{\frac{2d}{d+1}}\right)^{\frac{d+1}{2d}}
	\le C_d\|B\|,
\end{equation}
where $\|B\|$ is defined in a similar way as above and the exponent $2d/(d+1)$ is optimal. The inequalities \eqref{ineq:bh 1931} have played a key role in Bohnenblust and Hille's solution \cite{bohnenblust1931absolute} to {\em Bohr's strip problem} \cite{bohr13} concerning the convergence of Dirichlet series. 
Such multilinear form inequalities \eqref{ineq:bh 1931} and their polynomial variants (which we shall recall for Boolean cubes) are known as  {\em Bohnenblust--Hille inequalities}. 

Since then, Bohnenblust--Hille inequalities have been extended to different contexts. Recent years have seen great progress in improving the constants in Bohnenblust--Hille inequalities (e.g. $C_d$ in \eqref{ineq:bh 1931}) and this has led to the resolution of a number of open problems in harmonic analysis. See for example \cite{defant2011bohnenblust,defant2014bh,bayart2014bohr,defant19book} and references therein. 

In \cite{blei} Blei extended Bohnenblust--Hille inequalities to polynomials on Boolean cubes with dimension-free constants.
Recently, this result was revisited by Defant, Masty\l o and P\'erez \cite{defant2019fourier}. Moreover, they proved that the dimension-free Bohnenblust--Hille constants for Boolean cubes actually grow at most subexponentially in the degree. To state their results, recall that any function %complex-valued function $f$ on the $n$-dimensional Boolean cube 
$f:\{-1,1\}^n\to \com$ has the {\em Fourier--Walsh expansion}:
\begin{equation*}
	f(x)=\sum_{S\subset [n]}\widehat{f}(S)\chi_S(x),
\end{equation*}
where for each $S\subset [n]:=\{1,\dots,n\}$, $\widehat{f}(S)\in \com$ and
$$\chi_S(x):=\prod_{j\in S}x_j,\qquad x=(x_1,\dots, x_n).$$ 
The function $f$ is said to be {\em of degree at most $d$} if $\widehat{f}(S)=0$ whenever $|S|>d$; and is said to be {\em $d$-homogeneous} if $\widehat{f}(S)=0$ whenever $|S|\neq d$. Defant, Masty\l o and P\'erez proved  the following theorem (they considered real-valued functions but the proof works for complex-valued case as well). 

\begin{thm}\cite[Theorem 1]{defant2019fourier}\label{thm:discrete bh}
	For any $d\ge 1$, there exists $C_d>0$ such that for any $n\ge 1$ and any $f:\{-1,1\}^n\to \com$ of degree at most $d$, we have 
	\begin{equation}\label{ineq:bh discrete}
		\left(\sum_{|S|\le d}|\widehat{f}(S)|^{\frac{2d}{d+1}}\right)^{\frac{d+1}{2d}}\le C_d \|f\|_\infty.
	\end{equation} 
Denoting $\textnormal{BH}^{\le d}_{\{\pm 1\}}$ the best constant $C_d$ such that \eqref{ineq:bh discrete} holds, then there exists $C>0$ such that $\textnormal{BH}^{\le d}_{\{\pm 1\}}\le C^{\sqrt{d\log d}}$. So the Bohnenblust--Hille constant $\textnormal{BH}^{\le d}_{\{\pm 1\}}$ is of subexponential growth. 
\end{thm}

As we mentioned, Blei \cite{blei} first proved $\textnormal{BH}^{\le d}_{\{\pm 1\}}<\infty$. One of main contributions of Defant, Masty\l o and P\'erez \cite{defant2019fourier} lies in this subexponential bound $\textnormal{BH}^{\le d}_{\{\pm 1\}}\le C^{\sqrt{d\log d}}$. Recently, this Bohnenblust--Hille inequality for Boolean cubes \eqref{ineq:bh discrete} has found great applications in learning bounded low-degree functions on Boolean cubes \cite{eskenazis2022learning}, which will be explained in more detail in Section \ref{sect:learning}. In view of this, an analogue of \eqref{ineq:bh discrete} for qubit systems was conjectured in \cite{RWZ22}, motivated by learning quantum observables following the work of Eskenazis and Ivanisvili \cite{eskenazis2022learning}. But actually this quantum analogue of Bohnenblust--Hille inequality was already contained in a result of Huang, Chen and Preskill in the preprint \cite{CHP} that was not online available when the conjecture was made.  
%So they already resolved the conjecture about the Bohnenblust--Hille inequality in \cite{RWZ22}. 
Their motivation is to predict any quantum process. In this paper, we provide another proof that is simpler and gives better constants. Moreover, our method is more general which allows us to reduce many problems on the qubit systems to classical Boolean cubes. We refer to Section \ref{sect:reduction} for more discussions and to Section \ref{sect:diss} for the comparison on our results and the work in \cite{CHP}.

In our quantum setup, the Boolean cubes $\{-1,1\}^n$ are replaced by $M_2(\com)^{\otimes n}$, the $n$-fold tensor product of 2-by-2 complex matrix algebras. Recall that Pauli matrices and the identity matrix
\begin{equation*}
	\sigma_0=\begin{pmatrix}1&0\\0&1\end{pmatrix},\quad \sigma_1=\begin{pmatrix}0&1\\1&0\end{pmatrix},\quad \sigma_2=\begin{pmatrix}0&-i\\i&0\end{pmatrix},\quad \sigma_3=\begin{pmatrix}1&0\\0&-1\end{pmatrix},
\end{equation*}
 form a basis of $M_2(\com)$. For $\bs=(s_1,\dots, s_n)\in\{0,1,2,3\}^n$, we put
 \begin{equation*}
 	\sigma_\bs:=\sigma_{s_1}\otimes\dots\otimes \sigma_{s_n}\, .
 \end{equation*}
 All the $\sigma_{\bs}, \bs\in \{0,1,2,3\}^n$ form a basis of $M_2(\com)^{\otimes n}$ and play the role of characters $\chi_S,S\in [n]$ in the classical case. So any $A\in M_2(\com)^{\otimes n}$ has the unique Fourier expansion
 \begin{equation*}
 	A=\sum_{\bs\in\{0,1,2,3\}^n} \widehat{A}_\bs \,\sigma_\bs
 \end{equation*}
 with  $\widehat{A}_\bs\in \com$ being the Fourier coefficient. For any  $\bs=(s_1,\dots, s_n)\in\{0,1,2,3\}^n$, we denote by $|\bs|$ the number of non-zero $s_j$'s. Similar to the classical setting, $A\in M_2(\com)^{\otimes n}$ is {\em of degree at most $d$}   if $\widehat{A}_\bs =0$ whenever $|\bs|>d$, and it is {\em $d$-homogeneous} if $\widehat{A}_\bs=0$ whenever $|\bs|\neq d$. 
 
 In the sequel, we always use $\|A\|$ to denote the operator norm of $A$. Our main result is the following:
 
\begin{thm}\label{thm:quantum bh}
	For any $d\ge 1$, there exists $C_d>0$ such that for all $n\ge 1$ and all $\pol=\sum_{|\bs|\le d}\fc_{\bs}\sigma_{\bs}\in M_2(\com)^{\otimes n}$ of degree at most $d$, we have 
	\begin{equation}\label{ineq:quantum bh}
		\left(\sum_{|\bs|\le d}|\fc_{\bs}|^{\frac{2d}{d+1}}\right)^{\frac{d+1}{2d}}
		\le C_d \|\pol\|.
	\end{equation}
	Moreover, denote $\textnormal{BH}^{\le d}_{M_2(\com)}$ the best constant $C_d>0$. Then we have $\textnormal{BH}^{\le d}_{M_2(\com)}\le 3^d \textnormal{BH}^{\le d}_{\{\pm 1\}}$, so that it is at most of exponential growth.
\end{thm}

A special choice of  polynomials yields noncommutative analogues of Bohnenblust--Hille inequalities for multilinear forms. For this we use the following notation. Fix $n\ge 1$. For $\kappa\in\{1,2,3\}$ and $i\in [n]$, we write $\sigma^{(\kappa)}_i$ for $\sigma_{\bs}$ where $\bs=(s_1,\dots, s_n)\in\{0,1,2,3\}^n$ with $s_i=\kappa$ and $s_j=0$ whenever $j\neq i$.
\begin{cor}\label{cor:multilinear}
	Fix $d\ge 1$. Then there exists $C_d>0$ such that for any $n\ge 1$ and any (each $ \sigma^{(\kappa_j)}_{i_j}\in M_2(\com)^{\otimes n}$)
	$$A:=\sum_{\kappa_1,\dots, \kappa_d\in \{1,2,3\}}\sum_{i_1,\dots, i_d=1}^{n}a^{\kappa_1,\dots, \kappa_d}_{i_1,\dots, i_d}\sigma^{(\kappa_1)}_{i_1}\otimes \cdots \otimes \sigma^{(\kappa_d)}_{i_d},$$
	we have 
	\begin{equation*}
	\left(\sum_{\kappa_1,\dots, \kappa_d\in \{1,2,3\}}\sum_{i_1,\dots, i_d=1}^{n}|a^{\kappa_1,\dots, \kappa_d}_{i_1,\dots, i_d}|^{\frac{2d}{d+1}}\right)^{\frac{d+1}{2d}}
	\le C_d	\|A\|.
	\end{equation*}
	Moreover, $C_d\le 3^d \textnormal{BH}^{\le d}_{\{\pm 1\}}$, and it becomes a noncommutative analogue of Littlewood's 4/3 inequality when $d=2$.
\end{cor}

\begin{rem}
	Note that the algebra of function on $\{-1,1\}^n$ can be viewed as a commutative subalgebra of $M_2(\com)^{\otimes n}$ spanned by $\sigma_{\bs},\bs\in\{0,3\}^n$. So \eqref{ineq:bh discrete} is a special case of \eqref{ineq:quantum bh}
	%If we denote $\textnormal{BH}^{\le d}_{M_2(\com)}$ the best constant $C_d$ such that \eqref{ineq:quantum bh} holds, then 
	and we always have $\textnormal{BH}^{\le d}_{\{\pm 1\}}\le \textnormal{BH}^{\le d}_{M_2(\com)}$. Our main result Theorem \ref{thm:quantum bh} states that the converse holds up to a factor $3^d$.
\end{rem}

\begin{rem}
	The main theorem of Huang, Chen and Preskill \cite[Theorem 5]{CHP} is actually more general
	%a generalization of noncommutative Bohnenblust--Hille inequality, 
	which admits \eqref{ineq:quantum bh} as a corollary \cite[Corollary 3]{CHP}. Their proof is different from ours and the constant they obtained is $C_d\sim d^{\mathcal{O}(d)}$ which is worse. 
\end{rem}

Recall that a function $f:\{-1,1\}^n\to \com$ is called a {\em $k$-junta} if it depends on at most $k$ coordinates. Similarly, a matrix $A\in M_{2}(\com)^{\otimes n}$ is a {\em $k$-junta} if it acts non-trivially on at most $k$ qubits, that is, 
\begin{equation*}
		|\{1\le j\le n: \exists  \bs\in\{0,1,2,3\}^n \quad\textnormal{s.t.}\quad \widehat{A}_{\bs}\ne 0 \quad \& \quad s_j\neq 0\}|\le k. 
\end{equation*}
It is known that \cite{bourgain2002distribution,DFKO2007fourier} if a bounded  function $f$ over $\{-1,1\}^n$ is of low degree, then it is close to a junta. In the next corollary we derive such a result in a quantum setting. We refer to \cite[Theorem 3.9]{RWZ22} to another quantum junta type theorem related to the influences instead of the degree. 

\begin{cor}\label{cor:junta}
	Suppose that $A\in M_2(\com)^{\otimes n}$ is of degree at most $d$  and $\|A\|\le 1$. Then for any $\epsilon>0$, there exists a $k$-junta $B\in M_2(\com)^{\otimes n}$ such that 
	\begin{equation*}
		\|A-B\|_2\le \epsilon \qquad {\rm with }\qquad k\le \frac{d\left(\textnormal{BH}^{\le d}_{M_2(\com)}\right)^{2d}}{\epsilon^{2d}}.
	\end{equation*}
Here $\|\cdot\|_2$ denoted the Hilbert--Schmidt norm with respect to the normalized trace, that is, $\|A\|^2_2=2^{-n}\tr [A^\ast A]$.
In particular, we may choose $k\le dC^{2d^2}\epsilon^{-2d}$ for some universal $C>0$.
\end{cor}

\begin{rem}
	The results in \cite{bourgain2002distribution,DFKO2007fourier} are more general. However, in the case when polynomials are of low degree, the proof presented here that uses Bohnenblust--Hille inequalities is simpler. We are grateful to Alexandros Eskenazis for pointing out to us this proof. 
\end{rem}

To prove Theorem \ref{thm:quantum bh}, we reduce the problem to the (commutative) Boolean cube case, at a price of an extra factor $3^d$. In fact, our main contribution is a general method that reduces many problems in the quantum setting to the classical setting. This method will be explained in Section \ref{sect:reduction}, while the proof of Theorem \ref{thm:quantum bh} and Corollary \ref{cor:junta} will be presented  in Section \ref{sect:quantum bh}.

We shall illustrate the strength of this reduction method with two more applications. The first one concerns learning bounded low-degree quantum observables which will be Section \ref{sect:learning} and the main result is  Theorem \ref{thm:learning}. The second one is  Theorem \ref{thm:boolean radius} on Bohr's radius phenomenon in the context of quantum Boolean cubes which will be discussed in Section \ref{sect:bohr's radius}. %\textcolor{red}{\st{The similar idea leads to Theorem \ref{thm:learning} on learning low-degree quantum observables and Theorem \ref{thm:boolean radius} on Bohr's radius phenomenon in the context of quantum Boolean cubes, which will be discussed in Sections \ref{sect:learning} and \ref{sect:bohr's radius}, respectively.}} 
	
In Section \ref{sect:diss}, we briefly compare our results with the work of Huang, Chen and Preskill \cite{CHP}.

\medskip

\textbf{Notation}. We shall use $\tr$ to denote the usual (unnormalized) trace on matrix algebras, and $\langle \cdot, \cdot\rangle$ the inner product on $\com^n$ that is linear in the second argument. By  $\|A\|_p$ of a $k$-by-$k$ matrix $A$ we always mean the normalized Schatten-$p$ norm, i.e. $\|A\|_p^p=2^{-k}\tr |A|^p$. For any unit vector $\eta\in \com^n$, we use $\ket{\eta}\bra{\eta}$ to denote the associated rank one projection operator. Sometimes people use the convention $\eta\otimes \eta$ instead. By a density matrix we mean a positive semi-definite matrix with unit trace.

\subsection*{Acknowledgement} 
We are grateful to Hsin-Yuan Huang, Sitan Chen and John Preskill for sharing their preprint and for their helpful comments on an earlier version of this paper, especially Section \ref{sect:diss}.  We also would like to thank Alexandros Eskenazis for Corollary \ref{cor:junta} and the anonymous referee for valuable comments and remarks. 

\section{Reduction to the commutative case}
\label{sect:reduction}

In this section, we present a general reduction method. For this let us collect a few facts about Pauli matrices. For each $j=1,2,3$, $\sigma_j$ is self-adjoint unitary, and has $1$ and $-1$ as eigenvalues. We denote by $e^j_1$ and $e^j_{-1}$ the corresponding unit eigenvectors, respectively. Pauli matrices $\sigma_j, j=1,2,3$ satisfy the following anticommutation relation:
\begin{equation}\label{eq:anticommutative}
	\sigma_j\sigma_k+\sigma_k\sigma_j=0, \qquad j\neq k\in\{1,2,3\}.
\end{equation}
We record the following simple fact as a lemma. 

\begin{lem}\label{lem: fact}
	For any $j, k\in\{1,2,3\}$ and $\epsilon\in\{-1,1\}$, we have 
	\begin{equation}\label{eq:key}
		\langle \sigma_j e^{k}_{\epsilon},e^{k}_{\epsilon}\rangle=\delta_{jk}\epsilon.
	\end{equation}
\end{lem}

\begin{proof}
	When $j=k$, \eqref{eq:key} is trivial by definition of $e^k_\epsilon$. When $j\neq k$, by \eqref{eq:anticommutative} we have 
	\begin{equation*}
		\epsilon\langle \sigma_j e^{k}_{\epsilon},e^{k}_{\epsilon}\rangle
		=\langle \sigma_j \sigma_k e^{k}_{\epsilon},e^{k}_{\epsilon}\rangle
		=-\langle \sigma_k \sigma_j e^{k}_{\epsilon},e^{k}_{\epsilon}\rangle
		%=-\langle  \sigma_j e^{k}_{\epsilon},\sigma_ke^{k}_{\epsilon}\rangle
		=-\epsilon\langle \sigma_j e^{k}_{\epsilon},e^{k}_{\epsilon}\rangle.
	\end{equation*}
	This gives \eqref{eq:key} for $j\neq k$ since $\epsilon\neq 0$.
\end{proof}

Recall that $A\in M_2(\com)^{\otimes n}$  has the form
\begin{equation}\label{eq:fourier expansion quantum 1}
	\pol=\sum_{\bs\in \{0,1,2,3\}^n}\fc_{\bs}\sigma_{\bs}.
\end{equation}
It will be more convenient for us to rewrite it as 
\begin{equation}\label{eq:fourier expansion quantum 2}
	\pol=\sum_{l\ge 0}\sum_{\kappa_1,\dots, \kappa_l\in \{1,2,3\}}\sum_{1\le i_1<\cdots <i_l\le n}a^{\kappa_1,\dots,  \kappa_l}_{i_1,\dots, i_l}\sigma^{\kappa_1,\dots,  \kappa_l}_{i_1,\dots, i_l}
\end{equation}
where to each $(l;\kappa_1,\dots, \kappa_l;i_1<\cdots<i_l)$, we associate it with 
$\bs=(s_1,\dots, s_n)\in \{0,1,2,3\}^n$ of length $|\bs|=l$ with
\begin{equation*}
	s_{k}=
	\begin{cases*}
		\kappa_j,&$ k=i_j, 1\le j\le l$\\
		0,& \textnormal{otherwise}
	\end{cases*}
\end{equation*}
so that 
$$\fc_{\bs}=a^{\kappa_1,\dots,  \kappa_l}_{i_1,\dots, i_l}\qquad \textnormal{and} \qquad \sigma_{\bs}=\sigma^{\kappa_1,\dots,  \kappa_l}_{i_1,\dots, i_l}.$$
In other words, $\sigma^{\kappa_1,\dots,  \kappa_l}_{i_1,\dots, i_l}$ is defined as 
$$\sigma^{\kappa_1,\dots,  \kappa_l}_{i_1,\dots, i_l}:=\cdots\otimes  \sigma_{\kappa_1}\otimes \cdots \otimes \sigma_{\kappa_l}\otimes \cdots,$$
where $\sigma_{\kappa_j}$ appears in the $i_j$-th place for each $1\le j\le l$, and all other $(n-l)$ components are simply identity matrices $\sigma_0$.  

\begin{comment}

	The key ingredient of the proof is the following lemma. 	
	\begin{lem}\label{lem:key}
		Fix $n\ge 1$. For each 
		$$\vec{\epsilon}:=\left(\epsilon^{(1)}_1,\dots, \epsilon^{(1)}_n,\epsilon^{(2)}_1,\dots, \epsilon^{(2)}_n,\epsilon^{(3)}_1,\dots, \epsilon^{(3)}_n\right)\in \{-1,1\}^{3n},$$ 
		define the positive semi-definite matrix with unit trace:
		\begin{equation*}
			\rho(\vec{\epsilon}):=\rho_1\otimes \cdots \otimes \rho_n \in M_2(\com)^{\otimes n}, 
		\end{equation*}
		where for each $1\le j\le n$
		\begin{equation*}
			\rho_j:=\rho_j(\vec{\epsilon})=\frac{1}{3}\ket{e^1_{\epsilon^{(1)}_j}}\bra{e^1_{\epsilon^{(1)}_j}}+\frac{1}{3}\ket{e^{2}_{\epsilon^{(2)}_j}}\bra{e^2_{\epsilon^{(2)}_j}}+\frac{1}{3}\ket{e^3_{\epsilon^{(3)}_j}}\bra{e^3_{\epsilon^{(3)}_j}}.
		\end{equation*}
		Then for any $d\ge 1$ and any $A\in M_2(\com)^{\otimes n}$ of degree-$d$, there exists a polynomial $f_A:\{-1,1\}^{3n}\to \com$ of degree-$d$ such that 
		\begin{equation*}
			f_A(\vec{\epsilon})=\tr[A\rho(\vec{\epsilon})], \qquad \forall \vec{\epsilon}\in\{-1,1\}^{3n}.
		\end{equation*}
		In particular, $\|f_A\|_\infty\le \|A\|$.
	\end{lem}
\end{comment}

The main ingredient of our method is the following Proposition \ref{key prop}. To make the statement brief, we shall use the map
\begin{equation*}
	q=q_n:\{\bs\in\{0,1,2,3\}^n\}\to \{S\subset [3n]\},
\end{equation*}
that realizes the above identification of \eqref{eq:fourier expansion quantum 1} and \eqref{eq:fourier expansion quantum 2}.  If $\bs=\bf{0}$ is the $0$ vector, then we define $q(\bf{0}):=\emptyset$. Each $\bs=(s_1,\dots, s_n)\in \{0,1,2,3\}^n$ with $1\le |\bs|=l\le n$ is assigned to a tuple $(\kappa_1,\dots, \kappa_l;i_1<\cdots< i_l)\in  [3]^l\times [n]^l$ with $s_{i_j}=\kappa_j$. Then we define 
	\begin{equation*}
		q(\bs):=\{n(\kappa_j-1)+i_j:1\le j\le l\}=:S.
	\end{equation*}
	For example, if $n=5$ and $\bs =(0,1,2,3,1)$, then 
	\begin{equation*}
		q(\bs)=S=\{2, 5+3, 10+4, 5\}\subset [15].
	\end{equation*}
	Note that $|S|=|\bs |=l$. So 
	\begin{equation*}
		q\left(\{\bs\in\{0,1,2,3\}^n:|\bs|\le d\}\right)\subset \{S\subset [3n]:|S|\le d\}.
	\end{equation*}
	Moreover, this map $q:\{0,1,2,3\}^n\to \{S\subset [3n]\}$ is injective but not surjective. We denote by $p=p_n$ its inverse over $\im (q)\subset 2^{[3n]}$ (here and in what follows, we use $2^{[3n]}$ to denote the family of subsets of $[3n]$)
	$$
	\bs =p(S), \quad S\in \im (q)\subset 2^{[3n]}.
	$$
	Therefore, the above formulae \eqref{eq:fourier expansion quantum 1} and \eqref{eq:fourier expansion quantum 2} can also be rewritten as 
	\begin{equation}\label{eq:fourier expansion quantum 3}
		A=\sum_{S\in \im (q)\subset 2^{[3n]}} \widehat{A}_{p(S)}\sigma_{p(S)}.
	\end{equation}

	\begin{prop}\label{key prop}
		Fix $n\ge 1$. There exists a family 
		\begin{equation*}
			\mathcal{S}=\mathcal{S}_n=\{\rho(\vec{\epsilon}):\vec{\epsilon}\in \{-1,1\}^{3n}\}
		\end{equation*}
		of density matrices in $M_2(\com)^{\otimes n}$ such that the following holds: For any $A\in M_2(\com)^{\otimes n }$, the function $f_A:\{-1,1\}^{3n}\to \com$  defined by 
		\begin{equation*}
			f_A(\vec{\epsilon})=\tr [A\rho(\vec{\epsilon})],\qquad \vec{\epsilon}\in \{-1,1\}^{3n}
		\end{equation*}
		satisfies
		\begin{equation}\label{eq:key prop identity}
			\widehat{f}_A(S)=
			\begin{cases*}
				3^{-|S|}\widehat{A}_{p(S)}, & $S\in \im (q)\subset 2^{[3n]}$\\
				0,& $S\in 2^{[3n]}\setminus \im (q)$
			\end{cases*}.
		\end{equation}
		%\begin{enumerate}
		%\item $\|f_A\|_\infty\le \|A\|$;
		%\item and $\widehat{f}_A(S)=3^{-|S|}\widehat{A}_{p(S)}$ for all $S\in \im (q)$.
		%\end{enumerate}
		In particular, $\|f_A\|_\infty\le \|A\|$, $\deg(f_A)=\deg(A)$ and for all $d\ge 0$
		\begin{equation}\label{ineq:comparison of lp norm}
			\left(\sum_{\bs\in \{0,1,2,3\}^n:|\bs|\le d}|\widehat{A}_{\bs}|^{r}\right)^{1/r}\le 3^d\left(\sum_{S\subset [3n]:|S|\le d}	|\widehat{f}_A(S)|^{r}\right)^{1/r},  0< r<\infty.
		\end{equation}
	\end{prop}
	\begin{proof}
		For  any 
		$$\vec{\epsilon}:=\left(\epsilon^{(1)}_1,\dots, \epsilon^{(1)}_n,\epsilon^{(2)}_1,\dots, \epsilon^{(2)}_n,\epsilon^{(3)}_1,\dots, \epsilon^{(3)}_n\right)\in \{-1,1\}^{3n},$$ 
		the matrix $\rho(\vec{\epsilon})$ is defined as follows
		\begin{equation*}
			\rho:=\rho(\vec{\epsilon})=\rho_1\otimes \cdots \otimes \rho_n \in M_2(\com)^{\otimes n}, 
		\end{equation*}
		where for each $1\le j\le n$
		\begin{equation*}
			\rho_j:=\rho_j(\vec{\epsilon})=\frac{1}{3}\ket{e^1_{\epsilon^{(1)}_j}}\bra{e^1_{\epsilon^{(1)}_j}}+\frac{1}{3}\ket{e^{2}_{\epsilon^{(2)}_j}}\bra{e^2_{\epsilon^{(2)}_j}}+\frac{1}{3}\ket{e^3_{\epsilon^{(3)}_j}}\bra{e^3_{\epsilon^{(3)}_j}}.
		\end{equation*}
		Recall for any $\kappa\in \{1,2,3\}$ and $\epsilon\in\{-1,1\}$, $e^\kappa_{\epsilon}$ is a unit vector. Hence each $\rho_j(\vec{\epsilon})$ is positive semi-definite with trace $1$. So $\rho(\vec{\epsilon})$ is also a density matrix. 
		%Now let us verify the statements (1) and (2). The former follows immediately from the definition and H\"older's inequality. 
		To prove \eqref{eq:key prop identity}, let us employ the notation \eqref{eq:fourier expansion quantum 2}. Then for $\rho=\rho(\vec{\epsilon})$ defined as above 
		\begin{align*}
			\tr[\pol\rho]=\sum_{l\ge 0}\sum_{\kappa_1,\dots, \kappa_l\in \{1,2,3\}}\sum_{1\le i_1<\cdots <i_l\le n}a^{\kappa_1,\dots,  \kappa_l}_{i_1,\dots, i_l} \tr\left[\sigma^{\kappa_1,\dots,  \kappa_l}_{i_1,\dots, i_l}\rho\right].
		\end{align*}
		By definition,
		\begin{align*}
			\tr\left[\sigma^{\kappa_1,\dots,  \kappa_l}_{i_1,\dots, i_l}\rho\right]
			=&\tr[\sigma_{\kappa_1}\rho_{i_1}]\cdots \tr[\sigma_{\kappa_l}\rho_{i_l}]\prod_{j\notin\{ i_1,\dots, i_l\}}\tr[\sigma_0\rho_j]\\
			=&\tr[\sigma_{\kappa_1}\rho_{i_1}]\cdots \tr[\sigma_{\kappa_l}\rho_{i_l}].
		\end{align*}
		By Lemma \ref{lem: fact}, for each $1\le \kappa\le 3$ and $\epsilon\in\{-1,1\}$
		\begin{equation*}
			\tr\left[\sigma_{\kappa_j}\ket{e^{\kappa}_{\epsilon}}\bra{e^{\kappa}_{\epsilon}}\right]=\langle \sigma_{\kappa_j}e^{\kappa}_{\epsilon},e^{\kappa}_{\epsilon}\rangle=\epsilon\delta_{\kappa_j \kappa}.
		\end{equation*}
		Thus (recall that $\kappa_j\in\{1,2,3\}$)
		\begin{equation*}
			\tr[\sigma_{\kappa_j}\rho_{i_j}]
			=\frac{1}{3}\sum_{\kappa=1}^{3}\tr\left[\sigma_{\kappa_j}\ket{e^{\kappa}_{\epsilon^{(\kappa)}_{i_j}}}\bra{e^{\kappa}_{\epsilon^{(\kappa)}_{i_j}}}\right]
			=\frac{1}{3}\sum_{\kappa=1}^{3}\epsilon^{(\kappa)}_{i_j}\delta_{\kappa_j \kappa}
			=\frac{1}{3}\epsilon^{(\kappa_j)}_{i_j}.
		\end{equation*}
		So we have shown that 
		\begin{equation*}
			\tr\left[\sigma^{\kappa_1,\dots,  \kappa_l}_{i_1,\dots, i_l}\rho\right]
			=\tr[\sigma_{\kappa_1}\rho_{i_1}]\cdots \tr[\sigma_{\kappa_l}\rho_{i_l}]
			=\frac{1}{3^l}\epsilon^{(\kappa_1)}_{i_1}\cdots\epsilon^{(\kappa_l)}_{i_l},
		\end{equation*}
		and thus 
		\begin{align*}
			f_A(\vec{\epsilon})
			=\sum_{l\ge 0}\sum_{\kappa_1,\dots, \kappa_l\in \{1,2,3\}}\sum_{1\le i_1<\cdots <i_l\le n}\frac{1}{3^l}a^{\kappa_1,\dots,  \kappa_l}_{i_1,\dots, i_l} \epsilon^{(\kappa_1)}_{i_1}\cdots\epsilon^{(\kappa_l)}_{i_l}.
		\end{align*}
		This is nothing but \eqref{eq:key prop identity}. To see the rest of statements, note first that by definition and H\"older's inequality:
		\begin{equation*}
			|f_A(\vec{\epsilon})|\le \tr|\rho(\vec{\epsilon})|\cdot  \|A\|=\|A\|.
		\end{equation*}
		The desired $\deg(f_A)=\deg(A)$ and the inequality \eqref{ineq:comparison of lp norm} follow  immediately from \eqref{eq:key prop identity} and the fact that $|p(S)|=|S|$.
	\end{proof}

\section{Proof of Theorem \ref{thm:quantum bh} and Corollary \ref{cor:junta}}
\label{sect:quantum bh}

In this section we prove Theorem \ref{thm:quantum bh} and Corollary \ref{cor:junta}. 

\begin{proof}[Proof of Theorem \ref{thm:quantum bh}]
	Using the notation \eqref{eq:fourier expansion quantum 3} we need to prove 
	\begin{equation*}
		\left(\sum_{S\in \im (q)\subset 2^{[3n]}:|S|\le d}|\widehat{A}_{p(S)}|^{\frac{2d}{d+1}}\right)^{\frac{d+1}{2d}}\le 3^d \textnormal{BH}^{\le d}_{\pm 1}\|A\|
	\end{equation*}
for all 
$$A=\sum_{S\in \im (q)\subset 2^{[3n]}:|S|\le d} \widehat{A}_{p(S)}\sigma_{p(S)}\in M_2(\com)^{\otimes n}.$$
Let $f_A:\{-1,1\}^{3n}\to \com$ be  the function associated to $A\in M_2(\com)^{\otimes n}$ constructed in Proposition \ref{key prop}. So it is also of degree at most $d$. Then the desired result follows from the following chain of inequalities:
\begin{align*}
		\left(\sum_{S\in \im (q)\subset 2^{[3n]}:|S|\le d}|\widehat{A}_{p(S)}|^{\frac{2d}{d+1}}\right)^{\frac{d+1}{2d}}
		&\le 3^d\left(\sum_{S\subset [3n]:|S|\le d}	|\widehat{f}_A(S)|^{\frac{2d}{d+1}}\right)^{\frac{d+1}{2d}}\\
		&\le  3^d\textnormal{BH}^{\le d}_{\pm 1}\|f_A\|_\infty\\
		& \le 3^d\textnormal{BH}^{\le d}_{\pm 1}\|A\|,
\end{align*}
where in the first and last inequalities we used Proposition \ref{key prop}, and in the second inequality we used the Bohnenblust--Hille inequality on $\{-1,1\}^n$.
\end{proof}

Note that Corollary \ref{cor:multilinear} follows immediately from Theorem \ref{thm:quantum bh}. Now we prove Corollary \ref{cor:junta}.

\begin{proof}[Proof of Corollary \ref{cor:junta}]
	Fix $\eta >0$. Consider the following set
	\begin{equation*}
		\mathcal{A}_\eta:=\left\{\bs\in \{0,1,2,3\}^n:|\widehat{A}_\bs|>\eta\right\}.
	\end{equation*}
Recall that $\|A\|\le 1$. Then by Markov's inequality and noncommutative Bohenblust--Hille inequality (denoting $|\mathcal{A}_\eta|$ the cardinality of $\mathcal{A}_\eta$)
\begin{equation*}
	|\mathcal{A}_\eta|
	\le \eta^{-\frac{2d}{d+1}}\sum_{\bs\in\mathcal{A}_\eta}|\widehat{A}_{\bs}|^{\frac{2d}{d+1}}
	\le  \eta^{-\frac{2d}{d+1}}\left(\textnormal{BH}^{\le d}_{M_2(\com)}\right)^{\frac{2d}{d+1}}.
\end{equation*}
Then $B_\eta:=\sum_{\bs\in \mathcal{A}_\eta}\widehat{A}_{\bs}\sigma_{\bs}$ depends on at most $k$ coordinates with
\begin{equation*}
	k\le d|\mathcal{A}_\eta|
	\le d\eta^{-\frac{2d}{d+1}}\left(\textnormal{BH}^{\le d}_{M_2(\com)}\right)^{\frac{2d}{d+1}}.
\end{equation*}
Again, by noncommutative Bohenblust--Hille inequality and the fact that $\|A\|\le 1$
\begin{equation*}
	\|A-B_\eta\|_2^2
	=\sum_{\bs\notin \mathcal{A}_\eta}|\widehat{A}_\bs|^2
		\le \eta^{\frac{2}{d+1}}\sum_{\bs\notin \mathcal{A}_\eta}|\widehat{A}_\bs|^{\frac{2d}{d+1}}
	\le \eta^{\frac{2}{d+1}}\left(\textnormal{BH}^{\le d}_{M_2(\com)}\right)^{\frac{2d}{d+1}}.
\end{equation*}
Now choose $B=B_\eta$ with
$$
\eta:=\epsilon^{d+1}\left(\textnormal{BH}^{\le d}_{M_2(\com)}\right)^{-d}.
$$
Then $\|A-B\|_2\le \epsilon$, where $B$ is a $k$-junta with 
\begin{equation*}
	k\le d\eta^{-\frac{2d}{d+1}}\left(\textnormal{BH}^{\le d}_{M_2(\com)}\right)^{\frac{2d}{d+1}}
	=\epsilon^{-2d}d\left(\textnormal{BH}^{\le d}_{M_2(\com)}\right)^{2d}.
\end{equation*}
This finishes the proof, since $\textnormal{BH}^{\le d}_{M_2(\com)}\le C^d$ for some universal $C>0$ by Theorem \ref{thm:quantum bh}. 
\end{proof}

\section{Learning quantum observables of low degree}
\label{sect:learning}

Let us review the classical learning model first. Suppose that we want to learn a class $\mathcal{F}$ of functions on  $\{-1,1\}^n$ using the random query model which we shall explain. For $N\ge 1$, let $X_1,\dots, X_N$ be $N$ i.i.d. random variables uniformly distributed on $\{-1,1\}^n$. Then how many random queries do we need to recover functions $f\in \mathcal{F}$ nicely? More precisely, fix the error parameters $\epsilon,\delta \in (0,1)$, what is the least number $N=N(\mathcal{F},\epsilon,\delta)>0$ such that for any $f\in \mathcal{F}$ together with $N$ random queries 
\begin{equation*}
	(X_1,f(X_1)),\dots, (X_N, f(X_N))
\end{equation*}
one can construct a random function $h:\{-1,1\}^n\to \mathbb{R}$ such that 
\begin{equation*}
	\|f-h\|_2^2\le \epsilon
\end{equation*}
with probability at least $1-\delta$ ? Here $\|f\|_2$ denotes the $L^2$-norm of $f$ with respect to the uniform probability measure on $\{-1,1\}^n$. When $\mathcal{F}=\mathcal{F}^{\le d}_n$ consists of functions $f:\{-1,1\}^n\to [-1,1]$ of degree at most $d$, Eskenazis and Ivanisvili \cite{eskenazis2022learning} proved that
$$
N(\mathcal{F}^{\le d}_n,\epsilon,\delta)\le \frac{1}{\eps^{d+1}} \Big(\log\frac{n}{\delta}\Big) C(d),
$$
where $C(d)$ depends on the Bohnenblust--Hille constant $\textnormal{BH}^{\le d}_{\{\pm 1\}}$ for functions on $\{-1,1\}^n$ of degree at most $d$. So logarithmic number $\mathcal{O}_{\epsilon,\delta,d}(\log(n))$ of random queries is sufficient to learn bounded low-degree polynomials in the above sense. This improves significantly the previous work \cite{LMN93,IRRRY21} on the dimension-dependence of $N(\mathcal{F}^{\le d}_n,\epsilon,\delta)$. Later on, Eskenazis, Ivanisvili and Streck \cite{EIS} proved that $\mathcal{O}_{\epsilon,\delta,d}(\log(n))$ is also necessary.
\medskip

Now suppose that we want to learn quantum observable $A$ over $n$-qubits, i.e. $A\in M_2(\com)^{\otimes n}$, of degree at most $d$ with $\|A\| \le 1$.
%\begin{equation*}
	%\pol=\sum_{|\bs|\le d}\fc_{\bs}\sigma_{\bs}
	%=\sum_{0\le l\le d}\sum_{\kappa_1,\dots, \kappa_l\in \{1,2,3\}}\sum_{1\le i_1<\cdots <i_l\le n}a^{\kappa_1,\dots,  \kappa_l}_{i_1,\dots, i_l}\sigma^{\kappa_1,\dots,  \kappa_l}_{i_1,\dots, i_l}.
%\end{equation*}
%Suppose we also know 
%\begin{equation}
	%\label{bdd}
	%\|A\| \le 1.
%\end{equation}
Our learning model is similar to the classical setting. The random queries are now replaced with 
\begin{equation*}
	\tr [A\rho],\qquad \rho\sim \mathcal{S},
\end{equation*}
where $\rho$ samples uniformly in some set $\mathcal{S}$ of density matrices in $M_2(\com)^{\otimes n}$. Our hope is to
%To learn it we can randomly choose a state (somehow), sampling it by the same law. After that we 
build another (random) observable $\widetilde A$ out of $N$ random queries such that
\begin{equation}
	\label{eps}
	\|\widetilde A -A\|_2^2\le  \eps
\end{equation}
with probability at least $1-\delta$.
%, where $\Tr:= \tr/2^n$is the normalized trace.
 Again, the question is, how many random queries  $N=N(\eps, \delta, d, n)$ do we need to accomplish this, and how does $N$ depend on $n$?

In the remaining part of this section we provide one answer to this question with $\mathcal{S}=\mathcal{S}_n$ constructed in Proposition \ref{key prop}.

%In the scalar case this was solved \textcolor{blue}{by Eskenazis and Ivanisvili} \cite{eskenazis2022learning} with
%$$
%N\le \frac{1}{\eps^{d+1}} \Big(\log\frac{n}{\delta}\Big) C(d),
%$$
%where $C(d)$ depends on the Bohnenblust--Hille constant $\textnormal{BH}^{\le d}_{\{\pm 1\}}$ for  polynomials on Boolean cubes $\Om_n:=\{-1,1\}^n$ of degree at most $d$. \textcolor{blue}{So logarithmic number of random queries is sufficient to learn bounded low-degree polynomials. Later on, Eskenazis, Ivanisvili and Streck \cite{EIS} proved that $\mathcal{O}(\log(n))$ is also necessary.}

 %The fact that $A$ is of degree at most $d$ might be not so important as remarked in the discussion before \cite[Theorem 4]{CHP}: the contribution of Pauli monomials  is exponentially decaying in the number of qubits that the Pauli monomials act nontrivially on.

%\subsection{Random queries and reducing to commutative case}
%\label{rq}

\begin{thm}\label{thm:learning}
	Suppose that $A\in M_2(\com)^{\otimes n}$ is of degree at most $d$ and $\|A\|\le 1$. Fix $\delta,\epsilon\in (0,1)$ and 
	\begin{equation*}
		N\ge \frac{C^{d\sqrt{d\log d}}}{\epsilon^{d+1}}\log\left(\frac{n}{\delta}\right),
	\end{equation*}
	with $C>0$ large enough. Then given any $N$ random density matrices $\rho^{(m)},1\le m\le N$ independently and uniformly sampled in $\mathcal{S}_n$, as well as random queries 
	\begin{equation*}
		\left(\rho^{(m)}, \tr[A\rho^{(m)}]\right),\qquad\rho^{(m)}\sim \mathcal{S}_n
	\end{equation*}
	we can construct a random $\widetilde{A}\in M_2(\com)^{\otimes n}$ such that 
	$\|A-\widetilde{A}\|_2^2\le \epsilon$ with probability at least $1-\delta$. %Here for each $\vec{x}\in \{-1,1\}^{3n}$, $\rho(\vec{x})$ is an explicit positive semi-definite matrix with trace $1$, independent of $A$. 
\end{thm}

\begin{proof}
Again, we will reduce the problem to commutative case, which we shall explain how. To any 
\begin{equation*}
	A=\sum_{S\in \im (q)\subset 2^{[3n]}:|S|\le d} \widehat{A}_{p(S)}\sigma_{p(S)}\in M_2(\com)^{\otimes n},
\end{equation*}
we associate it with the function $f_A:\{-1,1\}^{3n}\to \com$ constructed in Proposition \ref{key prop}. Recall that $f_A$ is of degree at most $d$, $\|f_A\|_\infty\le 1$, and 
\begin{equation*}
	f_A=\tr[A\rho(\cdot )]=\sum_{S\in \im (q)\subset 2^{[3n]}:|S|\le d}3^{-|S|} \widehat{A}_{p(S)}\chi_S.
\end{equation*}
Suppose that we have $\rho(\vec{x}(m))\in \mathcal{S}_n,1\le m\le N$  as our independent random density matrices and
\begin{equation*}
	\left(\rho(\vec{x}(m)), \tr[A\rho(\vec{x}(m))]\right),\qquad 1\le m\le N
\end{equation*}
as random queries, where $\vec{x}(m),1\le m\le N$ are i.i.d. random variables uniformly distributed on $\{-1,1\}^{3n}$. Similar to the commutative setting \cite{EIS}, we approximate $$\widehat f_A(S)=3^{-|S|} \widehat{A}_{p(S)},\qquad S\in \im (q),\quad |S|\le d$$ 
with the empirical Walsh coefficients
$$
\alpha_S:= \frac1N \sum_{m=1}^N f_A(\vec x(m)) \chi_S(\vec x(m))
=\frac1N \sum_{m=1}^N\tr[A\rho(m)] \chi_S(\vec x(m)).
$$
Of course $\bE \alpha_S = \widehat f_A(S)$ where $\bE$ is with respect to the uniform distribution.
Since $\alpha_S$ is the sum of i.i.d. random variables, we get by $\|f_A\|_\infty\le 1$ and the Chernoff bound that for $b>0$
\begin{equation}
	\label{chernoff}
	\bP\left\{ |\alpha_S- \widehat f_A(S)| > b\right\}  \le 2\exp\left(-N b^2/2\right), \forall S\in \im(q), |S|\le d\,.
\end{equation} 
Note that 
\begin{align*}
	|\{S\in \im(q):|S|\le d\}|
	&=|\{\bs\in\{0,1,2,3\}^n:|\bs|\le d\}|\\
	=&\sum_{l=0}^{d}3^l\binom{n}{l}
	\le 3^d\sum_{l=0}^{d}\binom{n}{l}.
\end{align*}
Then by \eqref{chernoff} and the union bound
\begin{align*}
&	\bP\left\{ \exists S\in \im(q), |S|\le d: |\alpha_S-\widehat f_A(S)|> b \right\} \\
	\le &2e^{-Nb^2/2}	|\{S\in \im(q):|S|\le d\}|\\
	\le &2e^{-Nb^2/2}3^d \sum_{l=0}^{d}{n \choose l}.
\end{align*}

Choosing 
\begin{equation}
	\label{N}
	N\ge \frac{2}{b^2}\log\left(\frac{2\cdot 3^d}{\delta}\sum_{l=0}^{d}\binom{n}{l}\right),
	%=\frac{C}{b^2}\log\Big(\frac{3^d n^d}{\delta} \Big)
\end{equation}
one achieves
\begin{equation}\label{Chernoff}
	\bP\left\{ |\alpha_S- \widehat f_A(S)| \le  b,\forall S\in \im(q), |S|\le d \right\} 
	%\ge 1-2e^{-Nb^2/2}3^d \sum_{l=0}^{d}{n \choose l} 
	\ge 1-\delta.
\end{equation}

\medskip

We continue to copycat \cite{eskenazis2022learning} and  introduce the random sets
$$\mathscr{S}_b:=\left\{S\in \im(q),|S|\le d: |\alpha_S|\ge 2b\right\}.$$
In view of \eqref{Chernoff}, with probability $\ge 1-\delta$:
 \begin{equation}\label{ineq:estimates of fc}
 	 \begin{cases*}
 		|\widehat{f}_A(S)|\le |\alpha_S|+|\alpha_S-\widehat{f}_A(S)|< 3b, & if  $S\notin \mathscr{S}_b$\\
 			|\widehat{f}_A(S)|\ge |\alpha_S|-|\alpha_S-\widehat{f}_A(S)|\ge b, &  if $S\in \mathscr{S}_b$
 	\end{cases*}.
 \end{equation}
The second line of \eqref{ineq:estimates of fc}, together with $\|f_A\|_\infty\le 1$ and the commutative Bohnenblust--Hille inequality \eqref{ineq:bh discrete}, yields
\begin{align*}
	|\mathscr{S}_b|\le b^{-\frac{2d}{d+1}}\sum_{S\in \mathscr{S}_b}|\widehat{f}_A(S)|^{\frac{2d}{d+1}}
	\le b^{-\frac{2d}{d+1}}\left(\textnormal{BH}^{\le d}_{\{\pm 1\}}\right)^{\frac{2d}{d+1}}.
\end{align*}
Fix $b$ as above. Consider the random polynomial in $M_2(\com)^{\otimes n}$

$$A_{b}:=\sum_{S\in \mathscr{S}_b}3^{|S|}\alpha_S\sigma_{p(S)}
=\sum_{q(\bs)\in \mathscr{S}_b}3^{|\bs|}\alpha_{q(\bs)}\sigma_{\bs}.$$
All combined, we have with probability at least $1-\delta$ that
\begin{align*}
	\|A-A_{b}\|_2^2
	\le &3^{2d}\sum_{S\in \mathscr{S}_b}|\alpha_S-\widehat{f}_A(S)|^2 +3^{2d}\sum_{S\notin \mathscr{S}_b}|\widehat{f}_A(S)|^2\\
	\le & 3^{2d} b^2|\mathscr{S}_b|+3^{2d}(3b)^{\frac{2}{d+1}}\sum_{S\notin \mathscr{S}_b}|\widehat{f}_A(S)|^{\frac{2d}{d+1}}\\
	\le &\left(3^{d+1}\textnormal{BH}^{\le d}_{\{\pm 1\}}\right)^{\frac{2d}{d+1}}\left(b^2\cdot b^{-\frac{2d}{d+1}}+(3b)^{\frac{2}{d+1}}\right)\\
	\le & 10\left(3^{d+1}\textnormal{BH}^{\le d}_{\{\pm 1\}}\right)^{\frac{2d}{d+1}}b^{\frac{2}{d+1}}.  
\end{align*}
To get an error bound $\|A-A_{b}\|_2^2\le \epsilon$, we choose 
$$b=10^{-\frac{d+1}{2}}\left(3^{d+1}\textnormal{BH}^{\le d}_{\{\pm 1\}}\right)^{-d}\epsilon^{\frac{d+1}{2}}.$$
Inserting this into \eqref{N}, we choose $N$ such that 
$$N\ge \frac{2\cdot 10^{d+1}\left(3^{d+1}\textnormal{BH}^{\le d}_{\{\pm 1\}}\right)^{2d}}{\epsilon^{d+1}}\log\left(\frac{2\cdot 3^d}{\delta}\sum_{l=0}^{d}\binom{n}{l}\right).$$
Noting moreover that (see for example \cite{eskenazis2022learning})
\begin{equation*}
	\sum_{l=0}^{d}\binom{n}{l}\le \left(\frac{en}{d}\right)^d,
\end{equation*}
we may choose 
\begin{equation*}
	N\ge \frac{C^{d\sqrt{d\log d}}}{\epsilon^{d+1}}\log\left(\frac{n}{\delta}\right),
	%N\ge \frac{C^{d^2}\left(\textnormal{BH}^{\le d}_{\{\pm 1\}}\right)^{2d}}{\epsilon^{d+1}}\log\left(\frac{n}{\delta}\right),
\end{equation*}
for some $C>0$ large enough. Here, the Bohenblust--Hille constant $\textnormal{BH}^{\le d}_{\{\pm 1\}}$ is contained in $C^{d\sqrt{d\log d}}$.%Then given any $N$ random sampling of $(\vec{x}(m),\tr[A\rho(m)])$, 
With this choice of $N$, the random polynomial $\widetilde{A}:=A_{b}$ above satisfies 
$$	\|A-\widetilde{A}\|_2^2\le \epsilon,$$
with probability $\ge 1-\delta$.
\end{proof}

\section{Bohr's radius phenomenon on quantum Boolean cubes}
\label{sect:bohr's radius}
One important application of classical Bohnenblust--Hille inequalities is to study  {\em Bohr's radius problem} \cite{Bohr14}. The original problem \cite{BK97bohr} concerns the $n$-dimensional torus $\mathbb{T}^n$ with $\mathbb{T}=\{z\in\com:|z|=1\}$ and the exact asymptotic behaviour of Bohr's radius was obtained by Bayart, Pellegrino and  Seoane-Sep\'ulveda \cite{bayart2014bohr} using the polynomial version of Bohnenblust--Hille inequalities \eqref{ineq:bh 1931} with the best constants (denoted by $\textnormal{BH}^{\le d}_{\mathbb{T}}$) of subexponential growth in the degree $d$. See also \cite{defant2011bohnenblust}. A Boolean analogue  of the problem was studied by Defant, Masty\l o and P\'erez in \cite{defant18bohr}, where  Bohr's radius is replaced by  {\em Boolean's radius}.

\begin{defn}
	 {\em Boolean's radius} of a function $f:\{-1,1\}^n\to \real$ is the positive real number $\Br_n(f)$ such that 
	\begin{equation*}
		\sum_{S\subset [n]}|\widehat{f}(S)|\Br_n(f)^{|S|}=\|f\|_\infty.
	\end{equation*}
Given a class $\cF_n$ of functions on $\{-1,1\}^n$, the Boolean radius of $\cF_n$ is defined as 
\begin{equation*}
	\Br_n(\cF_n):=\inf\{\Br_n(f):f\in\cF_n\}.
\end{equation*}

\end{defn}

Of particular interests to us are the following four classes of functions
\begin{enumerate}
	\item $\cF_n(\all)$: all real functions on $\{-1,1\}^n$  with
	\begin{equation*}
		\Br_n(\all):=\Br_n(\cF_n(\all));
	\end{equation*}
	\item $\cF_n(\hom)$: all real homogeneous functions on $\{-1,1\}^n$ with
		\begin{equation*}
		\Br_n(\hom):=\Br_n(\cF_n(\hom));
	\end{equation*}
	\item $\cF_n(=d)$: all $d$-homogeneous real functions on $\{-1,1\}^n$ with
	\begin{equation*}
		\Br_n(=d):=\Br_n(\cF_n(=d));
	\end{equation*}
\item $\cF_n(\le d)$: all real functions on $\{-1,1\}^n$ of degree at most $d$  with
\begin{equation*}
	\Br_n(\le d):=\Br_n(\cF_n(\le d)).
\end{equation*}
\end{enumerate}

The problem of Boolean's radius is to determine the right order of decay of $\Br_n(\cF_n)$ as $n\to \infty$. Among others, Defant, Masty\l o and P\'erez proved the following.

\begin{thm}\label{thm:boolean radius}\cite[Theorems 2.1 \& 3.1 \& 4.1 and Corollary 3.2]{defant18bohr}
	For any $1\le d\le n$, we have
	\begin{enumerate}
		\item $\Br_n(\all)=2^{1/n}-1$ and thus $\lim\limits_{n\to\infty}n\Br_n(\all)=\log 2$;
		\item there exists $C>1$ such that for all 
		\begin{equation*}
			c_d n^{\frac{1}{2n}}\binom{n}{d}^{-\frac{1}{2d}}
			\le  	\Br_n(=d)
			\le C_d n^{\frac{1}{2n}}\binom{n}{d}^{-\frac{1}{2d}},
		\end{equation*}
		where 
		\begin{equation*}
			c_d=\frac{1}{d^{\frac{1}{2d}}C^{\sqrt{\frac{\log d}{d}}}} \qquad {\rm and } \qquad C_d=C^{\frac{1}{d}};
		\end{equation*}
	\item $\lim\limits_{n\to\infty}\sqrt{\frac{n}{\log n}}\Br_n(\hom)=1$;
	\item there exist $c'_d, C'_d>0$ such that 
	\begin{equation*}
		\frac{c'_d}{n^{\frac{1}{2}}}
		\le	\Br_n(\le d)
		\le \frac{C'_d}{n^{\frac{d-1}{2d}}}.
	\end{equation*}
	\end{enumerate}

\end{thm}

We refer to \cite{defant18bohr} for more discussions on the similarity and differences between the Boolean cube case and torus case. 

The concept of Boolean radius carries over to the quantum setting without any difficulties.

\begin{defn}
The {\em quantum Boolean radius} of self-adjoint $A\in M_2(\com)^{\otimes n }$ is the positive real number $\qBr_n(A)$ such that 
\begin{equation*}
	\sum_{\bs\in \{0,1,2,3\}^n}|\widehat{A}_\bs|\qBr_n(A)^{|\bs|}=\|A\|.
\end{equation*}	
Similarly, one may define for a class $\cF_n'$ of self-adjoint matrices in $M_2(\com)^{\otimes n }$
\begin{equation*}
	\qBr_n(\cF_n'):=\inf\{\qBr_n(A):A\in \cF_n'\}.
\end{equation*}
\end{defn}

\begin{rem}
In \cite{defant18bohr} only the real-valued functions were considered and some arguments rely on the real structure. We are not going to discuss the possible extension to the complex-valued functions here. It is for this reason we require  $A$ to be self-adjoint. In fact, when $A$ is self-adjoint, the function $f_A=\tr[A\rho(\cdot)]$ on $\{-1,1\}^{3n}$ constructed in Proposition \ref{key prop} is real-valued and has real Fourier coefficients, so that we can use the results in \cite{defant18bohr} directly. We leave the problem for general $A$ to future study. 
\end{rem}

If $\cF_n$ denotes one of the four classes of functions (1-4)  on $\{-1,1\}^n$ listed as above, then we use $\cF_n^q$ to denote the quantum counterpart of class of matrices in $M_2(\com)^{\otimes n}$. For example, if $\cF_n=\cF_n(\le d)$ is the class of  polynomials on $\{-1,1\}^n$ of degree at most $d$,  then $\cF_n^q=\cF^q_n(\le d)$ denotes the class of \textcolor{blue}{self-adjoint  $A\in M_2(\com)^{\otimes n}$ of degree at most $d$}. Then our main result on quantum Boolean radius is the following.

\begin{thm}
	For any $1\le d\le n$ and any $\cF_n$ of the four classes of functions (1-4) listed above, we have
	\begin{equation*}
		\qBr_n(\cF_n^q)\le 	\Br_n(\cF_n)\qquad {\rm and }\qquad \Br_{3n}(\cF_{3n})\le 3\qBr_{n}(\cF_{n}^q).
	\end{equation*}
\end{thm}

\begin{proof}
	The first inequality is trivial, as $\cF_n$ can be viewed as a subset of $\cF^q_n$ with all the relevant structures (norm, degree etc.) preserved. In fact, for any $f:\{-1,1\}^n\to \real$ with Fourier expansion
	\begin{equation*}
		f(x_1,\dots,x_n)=\sum_{0\le l\le n}\sum_{1\le i_1<\cdots<i_l\le n}
		a_{i_1,\dots, i_l}x_{i_1}\cdots x_{i_l},
	\end{equation*}
consider the \textcolor{blue}{self-adjoint} matrix
	\begin{equation*}
	A_f=\sum_{0\le l\le n}\sum_{1\le i_1<\cdots<i_l\le n}
	a_{i_1,\dots, i_l}\sigma_{i_1,\dots, i_l}^{3,\dots, 3}.
\end{equation*}
Clearly, $\deg(f)=\deg(A_f)$ and $A_f\in \mathcal{F}^q_n$ whenever $f\in \mathcal{F}_n$. Note that the canonical basis of $M_2(\com)$ are eigenvectors of $\sigma_3$ corresponding to eigenvalues $1,-1$, respectively. Let us denote this basis $\{e_1,e_{-1}\}$. Then under the basis $\{e_{x_1}\otimes \cdots \otimes e_{x_n}: x_j\in \{-1,1\},1\le j\le n\}$, $A_f$ is a diagonal matrix with diagonal entries being
\begin{equation*}
	\langle A_f(e_{x_1}\otimes \cdots \otimes e_{x_n}),e_{x_1}\otimes \cdots \otimes e_{x_n}\rangle 
	=\sum_{0\le l\le n}\sum_{1\le i_1<\cdots<i_l\le n}
	a_{i_1,\dots, i_l}x_{i_1}\cdots x_{i_l}.
\end{equation*}
So $\|f\|_\infty=\|A_f\|$. By definition, we have $\Br_n(f)=\qBr_n(A_f)$ and this proves the first inequality.

	To prove the second inequality, we appeal to our reduction method again. 
%\textcolor{red}{\st{one may argue as in the classical setting e.g. in \cite{defant18bohr} using noncommutative Bohnenblust--Hille inequalities \eqref{ineq:quantum bh}. But there is a more direct argument as we have used in the proof of Theorem \ref{thm:quantum bh}. The idea is, again, to reduce the problem to the classical Boolean cubes. }}
	In fact, take any $A=A^\ast\in \cF_n^q\subset M_2(\com)^{\otimes n}$. Consider the  function $f_A:\{-1,1\}^{3n}\to \real$ constructed in Proposition \ref{key prop} which belongs to $\cF_{3n}$ whenever $\cF_n$ is one of four aforementioned classes (1-4). Take the class (3) $\mathcal{F}_n(=d)$ for example. By construction, $f_A$ is $d$-homogeneous if $A$ is. Recall that $\|f_A\|_\infty\le\|A\|$, and by definition of $\qBr_n(A)$:
	\begin{equation*}
		\sum_{\bs}|\widehat{A}_\bs|\qBr_n(A)^{|\bs|}
		=\sum_{S\in \im (q)\subset 2^{[3n]}}|\widehat{A}_{p(S)}|\qBr_n(A)^{|S|}
		=\|A\|.
	\end{equation*}	
	Then for $f_A:\{-1,1\}^{3n}\to \real$, we have 
	\begin{align*}
		\sum_{S\in \im (q)\subset 2^{[3n]}}|\widehat{f}_A(S)|(3\qBr_n(A))^{|S|}
		=&\sum_{S\in \im (q)\subset 2^{[3n]}}|\widehat{A}_{p(S)}|\qBr_n(A)^{|S|}\\
	%	=&\sum_{0\le l\le d}\sum_{\kappa_1,\dots, \kappa_l\in \{1,2,3\}}\sum_{1\le i_1<\cdots <i_l\le n}3^{-l}|a^{\kappa_1,\dots,  \kappa_l}_{i_1,\dots, i_l}|(3\qBr_n(A))^{l}\\
	%	=&\sum_{0\le l\le d}\sum_{\kappa_1,\dots, \kappa_l\in \{1,2,3\}}\sum_{1\le i_1<\cdots <i_l\le n}|a^{\kappa_1,\dots,  \kappa_l}_{i_1,\dots, i_l}|\qBr_n(A)^{l}\\
		%=&\sum_{|\bs|\le d}|\widehat{A}_\bs|\qBr_n(A)^{|\bs|}\\
		=&\|A\|\ge \|f_A\|_{\infty}.
\end{align*}
Therefore, by definition of $\Br_{3n}(f_A)$:
\begin{equation*}
	\Br_{3n}(f_A)\le 3\qBr_n(A).
\end{equation*}
So for any  $A=A^\ast\in \cF_n^q$ we find $f_A\in \cF_{3n}$ such that the above inequality holds. Therefore we get
	$$\Br_{3n}(\cF_{3n})\le 3\qBr_n(\cF_n^q)$$
by definitions of $\Br_{3n}(\cF_{3n})$ and $\qBr_n(\cF_n^q)$. This  concludes the proof of the second inequality. 
\end{proof}

\begin{rem}
	It is still possible to improve our estimates using the reduction method. For example, we have shown that for the class $\mathcal{F}_n^q(\textnormal{all})$ of all self-adjoint matrices in $M_2(\com)^{\otimes n}$
	\begin{equation*}
		 \frac{2^{1/(3n)}-1}{3}\le 
		\qBr_n(\cF_n^q(\textnormal{all}))\le 2^{1/n}-1.
	\end{equation*} 
But we can actually prove 
	\begin{equation*}
	\frac{2^{1/n}-1}{9}\le
	\qBr_n(\cF_n^q(\textnormal{all}))\le 2^{1/n}-1.
\end{equation*} 
In fact, for any $A=A^\ast\in M_2(\com)^{\otimes n}$ with $\|A\|=1$, suppose $f_A:\{-1,1\}^{3n}\to [-1,1]$ is constructed as before. Then for any $\textbf{0}\neq \bs =p(S)\in \{0,1,2,3\}^n$, we have by \cite[Lemma 2.2]{defant18bohr} that 
\begin{equation*}
		|\widehat{A}_{\mathbf{0}}|+	3^{-|S|}|\widehat{A}_{p(S)}|
=|\widehat{f}_A(\emptyset)|+	|\widehat{f}_A(S)|
		\le 1.
\end{equation*}
So we get
\begin{align*}
	\sum_{\bs\in \{0,1,2,3\}^n}|\widehat{A}_{\bs}|r^{|\bs|}
	&\le |\widehat{A}_{\mathbf{0}}|+(1- |\widehat{A}_{\mathbf{0}}|)\sum_{\bs\neq \mathbf{0}}(3r)^{|\bs|}\\
	&\le |\widehat{A}_{\mathbf{0}}|+(1- |\widehat{A}_{\mathbf{0}}|)\sum_{k=1}^{n}\binom{n}{k}(9r)^{k}\\
	&= |\widehat{A}_{\mathbf{0}}|+(1- |\widehat{A}_{\mathbf{0}}|)\left((1+9r)^n-1\right).
\end{align*}
This establishes the bound $\qBr_n(\cF_n^q(\textnormal{all}))\ge 	\frac{2^{1/n}-1}{9}$ by definition. 
\end{rem}

%	In fact, let us take any $A\in \cF_n^q\subset M_2(\com)^{\otimes n}$ which is of degree at most $d$ for some $1\le d\le n$ so that 
%	\begin{equation*}
	%	\pol=\sum_{0\le l\le d}\sum_{\kappa_1,\dots, \kappa_l\in \{1,2,3\}}\sum_{1\le i_1<\cdots <i_l\le n}a^{\kappa_1,\dots,  \kappa_l}_{i_1,\dots, i_l}\sigma^{\kappa_1,\dots,  \kappa_l}_{i_1,\dots, i_l}.
	%\end{equation*}
	%\begin{align*}
	%&f_A\left(x^{(1)}_1,\dots, x^{(1)}_n,x^{(2)}_1,\dots, x^{(2)}_n,x^{(3)}_1,\dots, x^{(3)}_n\right)\\:=&\sum_{0\le l\le d}\sum_{\kappa_1,\dots, \kappa_l\in \{1,2,3\}}\sum_{1\le i_1<\cdots <i_l\le n}3^{-l}a^{\kappa_1,\dots,  \kappa_l}_{i_1,\dots, i_l}x_{i_1}^{(\kappa_1)}\cdots x_{i_l}^{(\kappa_l)},
	%\end{align*}
	%According to the proof of Theorem \ref{thm:quantum bh}, 
	%	\begin{equation*}
		%	\|f_A\|_\infty\le\|A\|. 
		%\end{equation*}

\section{Discussions}
\label{sect:diss}

We briefly compare our results with the work of Huang, Chen and  Preskill in \cite{CHP}. 
	
	\begin{enumerate}
		\item In the work \cite{CHP}, the noncommutative  Bohnenblust--Hille inequality \eqref{ineq:quantum bh} (\cite[Corollary 3]{CHP}) follows from a more general result \cite[Theorem 5]{CHP} which we do not have. The noncommutative Bohnenblust--Hille constant ($C_d$ in \eqref{ineq:quantum bh}) we obtained is of exponential growth $\sim C^d$, which is better than theirs $\sim d^{\mathcal{O}(d)}$. Remark that the best known bound for the (commutative) Boolean cubes is subexponential $\sim C^{\sqrt{d\log d}}$ \cite{defant2019fourier}. 
		
		\item Our proof of \eqref{ineq:quantum bh} is different from theirs. We use an argument that reduces the problem to the commutative case and the proof looks simpler,  while their proof that is more self-contained, combines several technical estimates that can be useful to other problems. 
		
		\item Our better Bohnenblust--Hille constant yields an immediate improvement for learning quantum observable up to a small prediction error in their work.  In particular, this allows us to remove the $\log \log (1 / \epsilon)$ factor in the exponent of \cite[Eq. (A17)]{CHP}. It also yields an immediate improvement in the sample complexity for learning arbitrary quantum process. This is because the learning is achieved by considering the unknown observable to be the observable after Heisenberg evolution under the unknown quantum process.
		
	%	\item Our learning Theorem \ref{thm:learning} concerns quantum observables of low degree. Their results \cite[Theorem 1 \& 4 \& 14]{CHP} are more general. In fact, their results work for general quantum observables and do not need low-degree assumption. This is because average prediction error causes exponential decay in the higher-degree Pauli coefficients of the observable $A$. So, the low-degree property is immediately induced.
\end{enumerate}

\subsection*{Data availability}

All data generated or analysed during this study are included in this article.

%\bibliographystyle{alpha}
%\bibliography{references}

\newcommand{\etalchar}[1]{$^{#1}$}

\end{document}